\documentclass[a4paper,11pt]{amsart}

\usepackage[english]{babel}
\usepackage[utf8]{inputenc}
\usepackage[T1]{fontenc}
\usepackage{lmodern}
\usepackage{ragged2e}
\usepackage{crimson}
\usepackage[upright]{fourier} \usepackage{baskervald}

\usepackage{amscd}
\usepackage{amssymb}
\usepackage{latexsym}
\usepackage{mathrsfs}
\usepackage{mathtools}
\usepackage{wasysym}
\usepackage{xy}
\usepackage{tikz}
\usetikzlibrary{cd,matrix,arrows}
\usetikzlibrary{backgrounds}

\usepackage{array}
\usepackage{wrapfig}
\usepackage{graphicx}
\usepackage{multirow}
\usepackage{xcolor}

\usepackage[colorlinks=true, citecolor=green, linkcolor=red, pagebackref=true]{hyperref}
\usepackage[alphabetic]{amsrefs}
\usepackage{upref}
\usepackage{caption}
\usepackage{bookmark}

	\newcommand{\N}{\ensuremath{\mathbb{N}}}
	\newcommand{\Z}{\ensuremath{\mathbb{Z}}}

	\newcommand{\C}{\ensuremath{\mathbb{C}}}

	\DeclareMathOperator{\B}{B}
	\DeclareMathOperator{\K}{K}
	
	\DeclareMathOperator{\Hy}{H}
	\DeclareMathOperator{\Ker}{Ker}
	\DeclareMathOperator{\im}{Im}
	\DeclareMathOperator{\tors}{tors}
	\DeclareMathOperator{\id}{id}
	\DeclareMathOperator{\et}{\text{\'et}}
	
	\newcommand{\Az}[1]{\mathscr{#1}}
	\DeclareMathOperator{\cl}{cl}
	\DeclareMathOperator{\RBn}{\widetilde{\beta}}
	\DeclareMathOperator{\esta}{stab}
	\DeclareMathOperator{\s}{s}
	\DeclareMathOperator{\per}{per}
	\DeclareMathOperator{\diag}{diag}
	\DeclareMathOperator{\proj}{proj}
	\DeclareMathOperator{\Tr}{T}%
	\DeclareMathOperator{\M}{M}
	\DeclareMathOperator{\U}{U}
	\DeclareMathOperator{\SU}{SU}
	\DeclareMathOperator{\PU}{PU}
	\DeclareMathOperator{\GL}{GL}
	\DeclareMathOperator{\PGL}{PGL}
	\DeclareMathOperator{\Br}{Br}
	\DeclareMathOperator{\End}{End}
	
	\newcommand{\tensor}{\otimes}
	\newcommand{\iso}{\cong}
	\DeclareMathOperator{\m}{mult}
	\newcommand{\mif}[3]{#1:#2 \longrightarrow #3}

\theoremstyle{plain}
\newtheorem{thm}{Theorem}[section] 
\newtheorem{theorem}[thm]{Theorem}
\newtheorem{lemma}[thm]{Lemma}
\newtheorem{corollary}[thm]{Corollary}
\newtheorem{proposition}[thm]{Proposition}

\theoremstyle{definition}
\newtheorem{definition}[thm]{Definition}
\newtheorem{notation}[thm]{Notation}

\theoremstyle{remark}
\newtheorem{remark}[thm]{Remark}
\newtheorem{claim}[thm]{Claim}
\begin{document}

\title{Decomposition of Topological Azumaya Algebras}

\author[Arcila-Maya]{Niny Arcila-Maya}

\address{Department of Mathematics, University of British Columbia, Vancouver~BC, V6T 1Z2 Canada}

\email{ninyam@math.ubc.ca}

\begin{abstract}
Let $\Az{A}$ be a topological Azumaya algebra of degree $mn$ over a CW complex $X$. We give conditions for the positive integers $m$ and $n$, and the space $X$ so that $\Az{A}$ can be decomposed as the tensor product of topological Azumaya algebras of degrees $m$ and $n$. Then we prove that if $m<n$ and the dimension of $X$ is higher than $2m+1$, $\Az{A}$ may not have such decomposition.

\smallskip
\noindent \textbf{AMS subject classification.} Primary 55P99; Secondary 55Q52, 55S45, 16H05.
\smallskip
\noindent \textbf{Keywords.} topological Azumaya algebra, projective unitary group.
\end{abstract}

\maketitle

\section{Introduction}

The classical theory of central simple algebras over a field was generalized by Azumaya \cite{Azu1951} and Auslander--Goldman \cite{AG1960} by introducing the concept of an Azumaya algebra over a local commutative ring and over an arbitrary commutative ring, respectively. This concept was generalized by Grothendieck \cite[1.1]{GroI1966} to the notion of a topological Azumaya algebra.

Grothendieck \cite[Section 2]{GroI1966} defined the notion of an Azumaya algebra over any locally-ringed topos $(X_{\et}, \Az{O}_{X})$ where  $X_{\et}$ is an \'etale topos of a scheme $X$, and the local ring is $\Az{O}_{X}$ is the structure sheaf of $X$.

\begin{definition}
A \textit{topological Azumaya algebra of degree $n$} over a topological space $X$ is a bundle of associative and unital complex algebras over $X$ that is locally isomorphic to the matrix algebra $\M_{n \times n}(\C)$ where $\C$ has its ordinary topology, \cite[1.1]{GroI1966}.
\end{definition}

Topological Azumaya algebras are classified by pointed homotopy classes of maps to $\B\PGL_{n}(\C)$, as there is a bijective correspondence
\begin{align*}\label{isoAzuPrin}
\begin{Bmatrix} 
\text{Isomorphism classes of topological}\\
\text{Azumaya algebras of degree $n$ over $X$}
\end{Bmatrix}
\leftrightarrow
\begin{Bmatrix} 
\text{Isomorphism classes of}\\
\text{principal $G$-bundles over $X$}
\end{Bmatrix},
\end{align*}
where $G$ is the topological group of automorphisms of $\M_{n \times n} (\C)$ as an algebra, \cite[8.2]{Steen1951}. The Skolem--Noether theorem asserts that this is $\PGL_{n}(\C)$; i.e., matrices acting by conjugation.

For brevity of notation we work with $\U_{n}$ instead of $\GL_{n}(\C)$. Our choice of notation does not affect our results because $\U_{n}$ included in $\GL_{n}(\C)$ as the maximal compact Lie subgroup is a deformation retract, in particular the inclusion is a weak equivalence. Hence, the homotopy type of $\U_{n}$ is that of $\GL_{n}(\C)$. The homotopy equivalence is more than an equivalence
of spaces, it upgrades to one of topological groups, hence of classifiying spaces.

Let $a$ and $m$ be positive integers. Let $\mu_{m}\subset \U_{am}$ be the cyclic subgroup of order $m$ consisting of scalar matrices $\zeta I_{am}$ for $\zeta$ an $m$-th root of unity. If we have a principal $\U_{am}/\mu_{m}$-bundle on a topological space $X$, then the quotient map $q:\U_{am}/\mu_{m}\rightarrow \PU_{am}$ gives rise by an extension of structure group to a principal $\PU_{am}$-bundle and therefore a topological Azumaya algebra of degree $am$.

The tensor product of complex algebras can be extended to topological Azumaya algebras by performing the operation fiberwise. The Brauer group of a topological space $X$ classifies topological Azumaya algebras on $X$ up to Brauer equivalence: $\Az{A}$ and $\Az{A}'$ are Brauer equivalent if there exist complex vector bundles $\Az{V}$ and $\Az{V}'$, and an isomorphism $\Az{A} \tensor \End(\Az{V}) \iso \Az{A}' \tensor \End(\Az{V}')$ of bundles of $\C$-algebras. If $X$ is a finite dimensional CW complex, then $\Br(X)\iso \Hy^{3}(X;\Z)_{\tors}$ the torsion part of the cohomology group $\Hy^{3}(X;\Z)$, \cite{GroI1966}. The order of a class $\alpha \in \Br(X)$ is called the \textit{period} of $\alpha$, and it is denoted by $\per(\alpha)$.

\begin{equation*}
\begin{tikzcd}[row sep=large,column sep=large]
& \B\U_{am}/\mu_{m} \arrow[d] &\\
X \arrow[ru, bend left, "\Az{A}"] \arrow[r,"\chi_{m}"] & \K(\Z/m,2) \arrow[r,"\RBn_{m}"]    & \K(\Z,3).
\end{tikzcd}
\end{equation*}

The \textit{Brauer class} of a map $\mif{\Az{A}}{X}{\B\U_{am}/\mu_{m}}$ is an element in $\Br(X)$ which will be denoted by $\cl(\Az{A})$.  It is defined as follows. Let  $\chi_{m}$ denote the composite of the projection of $\B\U_{am}/\mu_{m}$ on the the first non-trivial stage of its Postnikov tower, $\B\U_{am}/\mu_{m} \rightarrow \K(\Z/m,2)$, and the unreduced Bockstein map, $\RBn_{m}:\K(\Z/m,2)\rightarrow \K(\Z,3)$, as illustrated in the diagram above. Then $\cl(\Az{A})$ is equal to the composite $\RBn_{m}\circ\chi_{m}$.

\begin{remark}
For a deeper discussion on topological Azumaya algebras and the Brauer group of a topological space, we refer the reader to \cite{AWtwisted2014}.
\end{remark}

Saltman asked in \cite[page 35]{Sdiv1999} whether there is prime decomposition for Azumaya algebras under the tensor product operation, as there is for central simple algebras over a field. Antieau--Williams answered this question for topological Azumaya algebras in \cite[Corollary 1.3]{AW2x32014} by showing the following result:
\begin{theorem}
For $n>1$ an odd integer, there exist a 6-dimensional CW complex $X$ and a topological Azumaya algebra $\Az{A}$ on $X$ of degree $2n$ and period $2$ such that $\Az{A}$ has no decomposition $\Az{A}\iso\Az{A}_{2}\tensor\Az{A}_{n}$ for topological Azumaya algebras of degrees $2$ and $n$, respectively.
\end{theorem}

The aim of this paper is to provide conditions on a positive integer $n$ and a topological space $X$ such that a topological Azumaya algebra of degree $n$ on $X$ has a tensor product decomposition. The main result of this paper is the following theorem:
\begin{theorem}\label{main}
Let $m$ and $n$ be positive integers such that $m$ and $n$ are relatively prime and $m<n$. Let $X$ be a CW complex such that $\dim(X)\leq 2m+1$.

If $\Az{A}$ is a topological Azumaya algebra of degree $mn$ over $X$, then there exist topological Azumaya algebras $\Az{A}_{m}$ and $\Az{A}_{n}$ of degrees $m$ and $n$, respectively, such that $\Az{A}\iso \Az{A}_{m}\tensor\Az{A}_{n}$. 
\end{theorem}

Theorem \ref{main} is a corollary of a more general result. We prove in Theorem \ref{abpq/pq} that a map $X\rightarrow \B\U_{abmn}/\mu_{mn}$ can be lifted to $\B\U_{am}/\mu_{m}\times \B\U_{bn}/\mu_{n}$ when the dimension of $X$ is less than $2am+2$, the positive integers $a$, $b$, $m$ and $n$ are such that $am$ is relatively prime to $bn$, and $am<bn$. The proof of Theorem \ref{abpq/pq} relies significantly in the description of the homomorphisms induced on homotopy groups by the $r$-fold direct sum of matrices $\mif{\oplus^{r}}{\U_{n}}{\U_{rn}}$ in the range $\{0,1,\dots,2n+1\}$. We call this set ``\textit{the stable range}'' for $\U_{n}$.

This paper is organized as follows. The second section presents preliminaries on the effect of direct sum and tensor product operations on homotopy groups of compact Lie groups related to the unitary groups $\U_{n}$. The third section is devoted to the proof of Theorem \ref{abpq/pq}. We explain in Remark \ref{notuniq} why the decomposition in Theorem \ref{main} is not unique up to isomorphism.

\subsection*{Acknowledgements}
The author would like to express her deep gratitude to Ben Williams, her thesis advisor, for having proposed this research topic, pointing out relevant references, and having devoted a great deal of time to discuss details of the research with the author.

I gratefully acknowledge the anonymous referee for reading the paper carefully and providing thoughtful comments, many of which have resulted in changes to the revised version of the manuscript.

\subsection*{Notation}
Throughout this paper, all topological spaces will have the homotopy type of a CW complex. We fix basepoints for connected topological spaces, and for topological groups we take the identities as basepoints. We write $\pi_{i}(X)$ in place of $\pi_{i}(X,x_{0})$.

\section{Stabilization of operations on \texorpdfstring{$\U_{n}$}{Un}}

Let $m, n \in \N$, we consider the following matrix operations:
\begin{enumerate}
\item The \textit{direct sum of matrices}, $\mif{\oplus}{\U_{m}\times \U_{n}}{\U_{m+n}}$ defined by 
\[
A\oplus B=
\begin{pmatrix}
A & 0\\
0 & B
\end{pmatrix}.
\]

\item The \textit{$r$-fold direct sum}, $\mif{\oplus^{r}}{\U_{n}}{\U_{rn}}$ given by $A^{\oplus r}=\underbrace{A\oplus \cdots \oplus A}_{r\text{-times}}$.

\item The \textit{tensor product of matrices}, $\mif{\otimes}{\U_{m}\times\U_{n}}{\U_{mn}}$ defined by
\[
A\tensor B =
\begin{pmatrix}
a_{11}B & \cdots & a_{1m}B\\
\vdots & \ddots & \vdots \\
a_{m1}B & \cdots & a_{mm}B
\end{pmatrix},
\]
for $A=(a_{ij}) \in \U_{m}$.

\item The \textit{$r$-fold tensor product}, $\mif{\otimes^{r}}{\U_{n}}{\U_{n^{r}}}$ given by $A^{\tensor r} =\underbrace{A\tensor \cdots \tensor A}_{r\text{-times}}$.
\end{enumerate}

The homomorphisms of homotopy groups induced by the operations above will be denoted by $\oplus_{*}$, $\oplus^{r}_{*}$, $\otimes_{*}$ and $\otimes^{r}_{*}$, respectively.

We begin by recalling low degree homotopy groups of the unitary groups and the special unitary groups. The first homotopy groups of $\U_{n}$ can be calculated by using Bott periodicity. Bott proves in \cite{bott1958} that
\[
\pi_{i}(\U_{n})\iso
\begin{cases}
0 &\text{if $i<2n$ is even,}\\
\Z &\text{if $i<2n$ is odd,}\\
\Z/n! &\text{if $i=2n$.}
\end{cases}
\]

Since $\SU_{n}$ is the universal cover of $\U_{n}$, and there is a fibration $\SU_{n}\hookrightarrow \U_{n} \xrightarrow{\det} S^{1}$, it follows that
\[
\pi_{i}(\SU_{n})\iso
\begin{cases}
0 &\text{if $i=1$,}\\
\pi_{i}(\U_{n}) &\text{otherwise.}
\end{cases}
\]

We now compute the low degree homotopy groups of $\U_{am}/\mu_{m}$ and $\SU_{am}/\mu_{m}$. As $\SU_{am}$ is a simply connected $m$-cover of $\SU_{am}/\mu_{m}$ we have
\[
\pi_{i}(\SU_{am}/\mu_{m})\iso
\begin{cases}
\Z/m &\text{if $i=1$,}\\
\pi_{i}(\SU_{am}) &\text{otherwise}.
\end{cases}
\]

All columns as well as the two top rows of diagram \eqref{hg} are short exact. The nine-lemma implies that the bottom row is also short exact.

\begin{equation}\label{hg}
\begin{tikzcd}[column sep=large]
\mu_{m} \arrow[equal,r] \arrow[hookrightarrow,d] & \mu_{m} \arrow[r] \arrow[hookrightarrow,d] & \{1\} \arrow[d]\\
\SU_{am} \arrow[hookrightarrow,r] \arrow[twoheadrightarrow,d] & \U_{am} \arrow[twoheadrightarrow,r,"\det"] \arrow[twoheadrightarrow,d] & S^{1}  \arrow[equal,d]\\
\SU_{am}/\mu_{m} \arrow[r,"i"] & \U_{am}/\mu_{m}  \arrow[r,"\det"] & S^{1} .
\end{tikzcd}
\end{equation}

Therefore, $\pi_{i}(\U_{am}/\mu_{m})\iso \pi_{i}(\SU_{am}/\mu_{m})$ for all $i>1$. It remains to compute the fundamental group of $\U_{am}/\mu_{m}$. 

By exactness of the bottom row of diagram \eqref{hg}, the induced sequence on fundamental groups is exact,
\begin{equation}\label{fgU/m}
\begin{tikzcd}
0 \arrow[r] & \pi_{1}\left(\SU_{am}/\mu_{m}\right) \arrow[r,"i_{*}"] & \pi_{1}\left(\U_{am}/\mu_{m}\right)   \arrow[r,"\det_{*}"] & \pi_{1}\left(S^{1}\right)  \arrow[r] & 0.
\end{tikzcd}
\end{equation}

The map $\det:\U_{am}\rightarrow S^{1}$ has a section $t:S^{1}\rightarrow \U_{am}$ defined by 
\[
t(\omega)=%
\begin{pmatrix}
\omega & 0 \\
0 & I_{am-1}
\end{pmatrix}.
\]
The section $t$ is one of groups; in fact $\U_{n}$ is a semi-direct
product of $S^{1}$ by $\SU_{n}$. This section induces a section of $\det:\U_{am}/\mu_{m}\rightarrow S^{1}$, which we also denote by $t$,
\begin{equation}\label{secdet}
\begin{tikzcd}[column sep=large]
1 \arrow[r] & \SU_{am}/\mu_{m} \arrow[r,"i"] & \U_{am}/\mu_{m}  \arrow[r,"\det"] \arrow[r, leftarrow, bend right=35, dotted, "t"] & S^{1}  \arrow[r] & 1.
\end{tikzcd}
\end{equation}

Since $\pi_{1}\left(S^{1}\right)\iso \Z$, sequence \eqref{fgU/m} splits. We describe $\pi_{1}\left(\U_{am}/\mu_{m}\right)$ in terms of $i_{*}:\pi_{1}\left(\SU_{am}/\mu_{m}\right) \rightarrow \pi_{1}\left(\U_{am}/\mu_{m}\right)$ and $t_{*}:\pi_{1}\left(S^{1}\right) \rightarrow \pi_{1}\left(\U_{am}/\mu_{m}\right)$ as $\pi_{1}\left(\U_{am}/\mu_{m}\right)= \im i_{*} \oplus \im t_{*} \iso \Z/m \oplus \Z$. 

\subsection{Stabilization}

Let $m,n \in \N$ and $m \leq n$. Define the map 
\begin{equation*}
\begin{tikzcd}[row sep=tiny, column sep=normal]
\s:\U_{m} \arrow[r,rightarrow] & \U_{m+n}\\
A \arrow[r,mapsto] & A\oplus I_{n}.
\end{tikzcd}
\end{equation*}

The standard inclusion of unitary groups $\U_{n} \hookrightarrow \U_{n+1}$ is $2n$-connected.  Since the map $\s$ is equal to the consecutive composite of standard inclusions, it follows that $\s$ is $2m$-connected. Hence, $\s$ induces a surjection in degree $2m$ and an isomorphism on homotopy groups in degrees less than $2m$. 

\begin{notation}
Let $\esta$ denote $\pi_{i}(\s)$, the homomorphism $\s$ induces on homotopy groups. Henceforth, the following isomorphism for $i<2m$ will be needed throughout the paper
\begin{equation}\label{ident}
\begin{tikzcd}[row sep=tiny, column sep=normal]
\esta:\pi_{i}(\U_{m}) \arrow[r,rightarrow,"\iso"] & \pi_{i}(\U_{m+n})
\end{tikzcd}
\end{equation}
to identify $\pi_{i}(\U_{m})$ with $\pi_{i}(\U_{m+n})$ for all $i<2m$.
\end{notation}

\begin{lemma}\label{techl}
Let $r:\U_{m}\rightarrow \U_{m}$ be conjugation by $P \in \U_{m}$. There is a basepoint preserving homotopy $H$ from $r$  to $\id_{\U_{m}}$ such that for all $t \in [0,1]$, $H(-,t)$ is a homomorphism.
\end{lemma}
\begin{proof}
Since $\U_{m}$ is path-connected, there exists a path $\alpha$ from $P$ to $I_{m}$ in $\U_{m}$. Define $H:\U_{m}\times [0,1] \rightarrow \U_{m}$ by $H(A,t)=\alpha(t)A\alpha(t)^{-1}$. Observe that $H(-,t):\U_{m}\to \U_{m}$, $A\mapsto H(A,t)$ is a homomorphism. Moreover, $H$ is such that 
\begin{align*}
H(I_{m},t)=I_{m}, \;\; H(A,0)=r(A) \;\text{ and }\; H(A,1)=A.
\end{align*}
Therefore, the result follows.
\end{proof}

\begin{lemma}\label{sj}
Let $n,r \in \N$. For all $j=1,\dots,r$ define $\mif{\s_{j}}{\U_{n}}{\U_{rn}}$ by 
\[
\s_{j}(A)=\diag(I_{n},\dots,I_{n},A,I_{n},\dots,I_{n}),
\]
where $A$ is in the $j$-th position. The maps $\s_{j}$ and $\s_{j+1}$ are pointed homotopic for all $j=1,\dots,r-1$.
\end{lemma}
\begin{proof}
The block matrix
\[
P_{j}=
\begin{pmatrix}
I_{(j-1)n} &  &  &  \\
 & 0 & I_{n} & \\
 & I_{n} & 0 & \\
 & & & I_{(r-j-1)n}
\end{pmatrix}
\]
is such that $P_{j}P_{j}=I_{rn}$ for $j=1,\dots,r-1$. Moreover, if $A, B \in \U_{n}$, then
\[
P_{j}\diag(I_{n},\dots,I_{n},A,B,I_{n},\dots,I_{n})P_{j}=\diag(I_{n},\dots,I_{n},B,A,I_{n},\dots,I_{n}),
\]
where $A$ and $B$ are in positions $(j,j)$, $(j+1,j+1)$, and $(j+1,j+1)$, $(j,j)$, respectively.

From Lemma \ref{techl}, $\s_{j}$ and $\s_{j+1}$ are pointed homotopic.
\end{proof}

\begin{notation}
We call the $\s_{j}$ maps \textit{stabilization maps}. As $\s_{1}$ is equal to $\s:\U_{n} \to \U_{n+(r-1)n}$, it follows that $\s_{j}$ is $2n$-connected for all $j=1,\dots,r$. From Lemma \ref{sj} the homomorphisms induced on homotopy groups by the stabilization maps are equal, hence $\esta$ also denotes $\pi_{i}(\s_{1})=\cdots=\pi_{i}(\s_{r})$.   Thus we identify $\pi_{i}(\U_{n})$ with $\pi_{i}(\U_{rn})$ for $i<2n$ through $\esta$. The identification allows one to introduce a slight abuse of notation, namely to identify $x$ and $\esta(x)$ for $x\in \pi_{i}(\U_{n})$ and $i<2n$.
\end{notation}

\subsection{Operations}

\begin{proposition}\label{DS}
Let $i\in \N$, the homomorphism $\mif{\oplus_{*}}{\pi_{i}(\U_{m})\times \pi_{i}(\U_{n})}{\pi_{i}(\U_{m+n})}$ is given by
\[
\oplus_{*}(x,y)=\esta(x)+\esta(y)
\]
for $x \in \pi_{i}(\U_{m})$ and $y \in \pi_{i}(\U_{n})$.
\end{proposition}
\begin{proof}
It is enough to observe that the direct sum factors as
\begin{equation*}
\begin{tikzcd}[row sep=tiny, column sep=huge]
\U_{m}\times\U_{n} \arrow[r,rightarrow,"\s_{1}\times\s_{2}"] & \U_{m+n}\times\U_{m+n} \arrow[r,"\m"] &\U_{m+n}\\
(A,B) \arrow[r,mapsto] &%
\left(%
\begin{pmatrix}
A & 0\\
0 & I_{n}
\end{pmatrix},%
\begin{pmatrix}
I_{m} & 0\\
0 & B
\end{pmatrix}%
\right) \arrow[r,mapsto] &%
\begin{pmatrix}
A & 0\\
0 & B
\end{pmatrix}.
\end{tikzcd}
\end{equation*}
Thus $\oplus_{*}(x,y)=\m_{*}\circ(\esta\times\esta)(x,y)=\esta(x)+\esta(y)$, where the last equality is true by the Eckmann-Hilton argument, \cite[Theorem 1.6.8]{SpaAT2012}.
\end{proof}

\begin{corollary}\label{cDS}
If $m<n$ and  $i<2m$, then $\oplus_{*}(x,y)=x+y$ for $x \in \pi_{i}(\U_{m})$ and $y \in \pi_{i}(\U_{n})$.
\end{corollary}
\begin{proof}
Since $\s_{1}$ and $\s_{2}$ are $2m$-connected, the homomorphisms $\mif{\esta}{\pi_{i}(\U_{m})}{\pi_{i}(\U_{m+n})}$ and $\mif{\esta}{\pi_{i}(\U_{n})}{\pi_{i}(\U_{m+n})}$ are isomorphisms $i<2m$ and $i<2n$, respectively. We use these isomorphisms to identify source and target.

From Proposition \ref{DS}, $\oplus_{*}(x,y)=\esta(x)+\esta(y)=x+y$ for $i<2m$.
\end{proof}

\begin{proposition}\label{nBS}
Let $i\in \N$, the homomorphism $\mif{\oplus^{r}_{*}}{\pi_{i}(\U_{n})}{\pi_{i}(\U_{rn})}$ is given by
\[
\oplus^{r}_{*}(x)=r\esta(x)
\]
for $x \in \pi_{i}(\U_{n})$.
\end{proposition}
\begin{proof}
Let $\Delta:\U_{n} \rightarrow (\U_{n})^{\times r}$ denote the diagonal map. The $r$-block summation factors as 
\begin{equation*}
\begin{tikzcd}[row sep=tiny, column sep=normal]
\U_{n} \arrow[r,rightarrow,"\Delta"] & (\U_{n})^{\times r} \arrow[r,rightarrow,"\s_{1}\times\cdots\times\s_{r}"] & (\U_{rn})^{\times r} \arrow[r,rightarrow,"\m"] & \U_{rn} \\
A \arrow[r,mapsto] & (A,\dots,A) \arrow[r,mapsto] & \Bigl(\s_{1}(A),\dots,\s_{r}(A)\Bigr) \arrow[r,mapsto] & \s_{1}(A)\cdots \s_{r}(A).
\end{tikzcd}
\end{equation*}
By the Eckmann--Hilton argument $\mif{\m_{*}}{\pi_{i}(\U_{rn})^{r}}{\pi_{i}(\U_{rn})}$ is given by 
\begin{equation*}
\m_{*}(x_{1},\dots,x_{r})=x_{1}+\cdots+x_{r}
\end{equation*}
for $x_{j} \in \pi_{i}(\U_{rn})$ and $j=1,\dots,r$. From this $\oplus_{*}^{r}$ takes the form
\begin{equation*}
\begin{tikzcd}
\pi_{i}(\U_{n})  & x \arrow[d,mapsto] \\
\pi_{i}(\U_{n})^{\times r} \arrow[u,leftarrow,"\Delta"] & (x,\dots,x) \arrow[d,mapsto]  \\
\pi_{i}(\U_{rn}) ^{\times r} \arrow[u,leftarrow,"\esta^{\times r}"] & \bigl(\esta(x),\dots,\esta(x)\bigr) \arrow[d,mapsto]\\
\pi_{i}(\U_{rn}) \arrow[u,leftarrow,"\m_{*}"] & \underbrace{\esta(x)+\cdots+\esta(x)}_{r \text{ times}}
\end{tikzcd}
\end{equation*} 
This proves the statement.
\end{proof}

\begin{corollary}\label{cnBS}
If $i<2n$, then $\oplus_{*}^{r}(x)=rx$ for $x \in \pi_{i}(\U_{n})$.
\end{corollary} 
\begin{proof}
The homomorphism $\mif{\esta^{\times r}}{\pi_{i}(\U_{n})^{\times r}}{\pi_{i}(\U_{rn})^{\times r}}$ is an isomorphism for all $i<2n$ because so is $\mif{\esta}{\pi_{i}(\U_{n})}{\pi_{i}(\U_{rn})}$. By Proposition \ref{nBS} we conclude $\oplus_{*}^{r}(x)=r\s_{*}(x)=rx$ for $i<2n$.
\end{proof}

\begin{lemma}\label{st}
Let $L, R: \U_{m}\rightarrow \U_{mn}$ be the maps $L(A)=A\tensor I_{n}$ and $R(A)=I_{n}\tensor A$. There is a basepoint preserving homotopy $H$ from $L$  to $R$ such that for all $t \in [0,1]$, $H(-,t)$ is a homomorphism.
\end{lemma}
\begin{proof}
Let $A\in \U_{m}$.
\[
L(A)=
\begin{pmatrix}
a_{11}I_{n}&  \cdots &  a_{1m}I_{n}\\
\vdots&  \ddots &  \vdots\\
a_{m1}I_{n}&  \cdots &  a_{mm}I_{n}
\end{pmatrix}%
\quad \text{and} \quad
R(A)=
\begin{pmatrix}
A&  \cdots &  0\\
\vdots&  \ddots &  \vdots\\
0&  \cdots &  A
\end{pmatrix}=A^{\oplus n}.
\]
Let $P_{m,n}$ be the permutation matrix
\begin{align*}
P_{m,n}=[&e_{1},e_{n+1},e_{2n+1},\dots, e_{(m-1)n+1},%
		      e_{2},e_{n+2},e_{2n+2},\dots, e_{(m-1)n+2},\\%
		   &\dots,\\
		   &e_{n-1},e_{2n-1},e_{3n-1},\dots,e_{mn-1},%
		     e_{n},e_{2n},e_{3n},\dots, e_{(m-1)n},e_{mn}]
\end{align*}
where, $e_{i}$ is the $i$-th standard basis vector of $\C^{mn}$ written as a column vector. Observe that $L(A)=P_{m,n}R(A)P_{m,n}^{-1}$. The result follows from Lemma \ref{techl}.
\end{proof}

\begin{proposition}\label{TP}
Let $i\in \N$, the homomorphism $\mif{\otimes_{*}}{\pi_{i}(\U_{m})\times\pi_{i}(\U_{n})}{\pi_{i}(\U_{mn})}$ is given by 
\[
\otimes_{*}(x,y)=n\esta(x)+m\esta(y)
\]
for $x \in \pi_{i}(\U_{m})$ and $y\in \pi_{i}(\U_{n})$.
\end{proposition}
\begin{proof}
By the mixed-product property of the tensor product of matrices 
\[
A\tensor B=(A\tensor I_{n})(I_{m}\tensor B)=L(A)R(B).
\]

Lemma \ref{st} gives $L_{*}=\oplus_{*}^{n}:\pi_{i}(\U_{m})\rightarrow \pi_{i}(\U_{mn})$. Proposition \ref{nBS} now yields $\otimes_{*}(x,y)=n\esta(x)+m\esta(y)$.
\end{proof}

\begin{corollary}\label{cTP}
If $m<n$ and $i<2m$, then $\otimes_{*}(x,y)=nx+my$ for $x \in \pi_{i}(\U_{m})$ and $y\in \pi_{i}(\U_{n})$.
\end{corollary}
\begin{proof}
The statement follows from Corollary \ref{cnBS} and Proposition \ref{TP}.
\end{proof}

\begin{proposition}\label{nBT}
Let $i\in \N$, the homomorphism $\mif{\otimes^{r}_{*}}{\pi_{i}(\U_{n})}{\pi_{i}(\U_{n^{r}})}$ is given by 
\[
\otimes^{r}_{*}(x)=rn^{r-1}\esta(x)
\]
for $x\in \pi_{i}(\U_{n})$.
\end{proposition}

\begin{corollary}\label{cnBT}
If $i<2n$, then $\otimes^{r}_{*}(x)=rn^{r-1}x$ for $x\in \pi_{i}(\U_{n})$.
\end{corollary}
\begin{proof}
Corollary \ref{cnBS} and Proposition \ref{nBT} yield the result.
\end{proof}

\subsubsection{Tensor product on the quotient}

Let $a$, $b$, $m$ and $n$ be positive integers so that $m<n$. The tensor product operation $\mif{\otimes}{\U_{am}\times\U_{bn}}{\U_{abmn}}$ sends the group $\mu_{m}\times \mu_{n}$ to $\mu_{mn}$. In consequence, the operation descends to the quotient 
\begin{equation}\label{q}
\mif{\otimes}{\U_{am}/\mu_{m}\times\U_{bn}/\mu_{n}}{\U_{abmn}/\mu_{mn}}.
\end{equation}

\begin{proposition}\label{TPi}
If $i>1$, the homomorphism 
\[
\mif{\otimes_{*}}{\pi_{i}(\U_{am}/\mu_{m})\times\pi_{i}(\U_{bn}/\mu_{n})}{\pi_{i}(\U_{abmn}/\mu_{mn})}
\]
is given by 
\[
\otimes_{*}(x,y)=bn\esta(x)+am\esta(y)
\]
for $x \in \pi_{i}(\U_{am}/\mu_{m})$ and $y \in \pi_{i}(\U_{bn}/\mu_{n})$.
\end{proposition}
\begin{proof}
There is a map of fibrations
\begin{equation}\label{mf}
\begin{tikzcd}
\mu_{m}\times\mu_{n} \arrow[r,hookrightarrow] \arrow[d,"\m"] & \U_{am}\times\U_{bn} \arrow[r] \arrow[d,"\otimes"] & \U_{am}/\mu_{m}\times\U_{bn}/\mu_{n} \arrow[d,"\otimes"]\\
\mu_{mn} \arrow[r,hookrightarrow] & \U_{abmn} \arrow[r] & \U_{abmn}/\mu_{mn}.
\end{tikzcd}
\end{equation}
From the homomorphism of long exact sequences associated to the fibrations in diagram \eqref{mf} we obtain a commutative square
\begin{equation*}
\begin{tikzcd}
\pi_{i}(\U_{am})\times\pi_{i}(\U_{bn}) \arrow[r,"\iso"] \arrow[d,"\otimes_{*}"] & \pi_{i}(\U_{am}/\mu_{m})\times\pi_{i}(\U_{bn}/\mu_{n}) \arrow[d,"\otimes_{*}"]\\
\pi_{i}(\U_{abmn}) \arrow[r,"\iso"] & \pi_{i}(\U_{abmn}/\mu_{mn}).
\end{tikzcd}
\end{equation*}
for $i>1$. This diagram and Proposition \ref{TP} gives $\otimes_{*}(x,y)=\oplus^{bn}_{*}(x)+\oplus^{am}_{*}(y)=bn\esta(x)+am\esta(y)$ for all $i>1$.
\end{proof}

In the following proposition, we identify $\pi_{1}(\U_{am}/\mu_{m})$ with $\im i_{*}\oplus \im t_{*}\iso\Z/m \oplus\Z$, where $i$ and $t$ are the maps in diagram \eqref{secdet}. We also identify $\Z/m$ and $\Z/n$ with the subgroups $\{0,n,2n,\dots,(m-1)n\}\subset \Z/mn$ and $\{0,m,2m,\dots,(n-1)m\}\subset \Z/mn$, respectively.

\begin{proposition}\label{TP1}
The homomorphism 
\[
\mif{\otimes_{*}}{\pi_{1}(\U_{am}/\mu_{m})\times\pi_{1}(\U_{bn}/\mu_{n})}{\pi_{1}(\U_{abmn}/\mu_{mn})}
\]
is given by 
\[
\otimes_{*}(\alpha+x,\beta+y)=(\alpha+\beta)+(bnx+amy)
\]
for $\alpha \in \Z/m \subset \Z/mn$, $\beta \in \Z/n \subset \Z/mn$, and $x, y \in \Z$.
\end{proposition}
\begin{proof}
Since the determinant of a tensor product is the product of powers of the determinants, we define $\phi:S^{1}\times S^{1} \rightarrow S^{1}$ by $\phi(\upsilon,\omega)=\upsilon^{bn}\omega^{am}$ so that the diagram below is a map of fibrations. 
\begin{equation*}
\begin{tikzcd}[column sep=large]
\SU_{am}/\mu_{m}\times\SU_{bn}/\mu_{n} \arrow[r,hookrightarrow,"i\times i"] \arrow[d,"\otimes"] & \U_{am}/\mu_{m}\times\U_{bn}/\mu_{n} \arrow[r,"\det\times\det"] \arrow[r, leftarrow, bend left=35, dotted, "t\times t"] \arrow[d,"\otimes"] & S^{1}\times S^{1} \arrow[d,"\phi"]\\
\SU_{abmn}/\mu_{mn} \arrow[r,hookrightarrow,"i"] & \U_{abmn}/\mu_{mn} \arrow[r,"\det"] \arrow[r, leftarrow, bend right=35, dotted, "t"] & S^{1}.
\end{tikzcd}
\end{equation*}
This map of fibrations induces a homomorphism of short exact sequences
\begin{equation}\label{hmf2}
\begin{tikzcd}[row sep=large]
\pi_{1}(\SU_{am}/\mu_{m})\times\pi_{1}(\SU_{bn}/\mu_{n}) \arrow[d,hookrightarrow,"i_{*}\times i_{*}"] \arrow[r,"\otimes_{*}"] & \pi_{1}(\SU_{abmn}/\mu_{mn}) \arrow[d,hookrightarrow,"i_{*}"] \\
 \pi_{1}(\U_{am}/\mu_{m})\times\pi_{1}(\U_{bn}/\mu_{n}) \arrow[d,twoheadrightarrow]  \arrow[r,"\otimes_{*}"]  & \pi_{1}(\U_{abmn}/\mu_{mn}) \arrow[d,twoheadrightarrow] \arrow[d, leftarrow, bend left=45, dotted, "t_{*}"]\\
\pi_{1}(S^{1})\times\pi_{1}(S^{1}) \arrow[u, bend left=45, dotted, "t_{*}\times t_{*}"] \arrow[r,"\phi_{*}"]  & \pi_{1}(S^{1}).
\end{tikzcd}
\end{equation}

We want to determine the homomorphism $\tensor_{*}$ in the middle of diagram \eqref{hmf2}. In order to do this, we will determine $\mif{\otimes_{*}}{\pi_{1}(\SU_{am}/\mu_{m})\times\pi_{1}(\SU_{bn}/\mu_{n})}{\pi_{1}(\SU_{abmn}/\mu_{mn})}$, and show that the short exact sequences in diagram \eqref{hmf2} split compatibly so that $\mif{\otimes_{*}}{\pi_{1}(\U_{am}/\mu_{m})\times\pi_{1}(\U_{bn}/\mu_{n})}{\pi_{1}(\U_{abmn}/\mu_{mn})}$ is equal to
\begin{equation*}
\begin{tikzcd}[row sep=large]
\Bigl(\pi_{1}(\SU_{am}/\mu_{m})\times\pi_{1}(\SU_{bn}/\mu_{n})\Bigr) \oplus \Bigl(\pi_{1}(S^{1})\times\pi_{1}(S^{1})\Bigr)  \arrow[d,"\otimes_{*}\;\oplus\;\phi_{*}"]\\
\pi_{1}(\SU_{abmn}/\mu_{mn}) \oplus \pi_{1}(S^{1}).
\end{tikzcd}
\end{equation*}

We begin by observing that there exists a similar map of fibrations to the one in diagram \eqref{mf}, but with the spaces $\SU_{am}$ and $\SU_{bn}$ instead of $\U_{am}$ and $\U_{bn}$, respectively. In this case we obtain the commutative square,
\begin{equation*}
\begin{tikzcd}
\pi_{1}(\SU_{am}/\mu_{m})\times\pi_{1}(\SU_{bn}/\mu_{n}) \arrow[r,"\iso"] \arrow[d,"\otimes_{*}"] & \Z/m \times \Z/n \arrow[d,"\psi"]\\
\pi_{1}(\SU_{abmn}/\mu_{mn}) \arrow[r,"\iso"] & \Z/mn,
\end{tikzcd}
\end{equation*}
where $\Z/m$ and $\Z/n$ are considered as subgroups of $\Z/mn$, and $\psi:\Z/m \times \Z/n \to \Z/mn$ is addition. From this $\otimes_{*}:\pi_{1}(\SU_{am}/\mu_{m})\times\pi_{1}(\SU_{bn}/\mu_{n})\rightarrow \pi_{1}(\SU_{abmn}/\mu_{mn})$ is equal to the addition. 

In order to prove the compatibility, we observe that even though diagram \eqref{hc} below does not commute, Claim \ref{chc} implies that it is commutative up to a pointed homotopy. Therefore, the induced diagram on homotopy groups does commute
\begin{equation*}
\begin{tikzcd}[column sep=huge]
\pi_{1}(S^{1})\times \pi_{1}(S^{1}) \arrow[r,"t_{*}\times t_{*}"] \arrow[d,"\phi_{*}"] & \pi_{1}(\U_{am}/\mu_{m})\times\pi_{1}(\U_{bn}/\mu_{n}) \arrow[d,"\otimes_{*}"] \\
\pi_{1}(S^{1}) \arrow[r,"t_{*}"] & \pi_{1}(\U_{abmn}/\mu_{mn}).
\end{tikzcd}
\end{equation*}

Consequently, the diagram below commutes
\begin{equation*}
\begin{tikzcd}
\pi_{1}(\U_{am}/\mu_{m})\times\pi_{1}(\U_{bn}/\mu_{n}) \arrow[d,"\otimes_{*}"] \arrow[r,"\iso"] &  \im (i_{*}\times i_{*}) \oplus \im (u_{*}\times u_{*}) \arrow[d,"\psi\oplus\phi_{*}"]\\
\pi_{1}(\SU_{abmn}/\mu_{mn}) \arrow[r,"\iso"] & \im i_{*} \oplus \im u_{*},
\end{tikzcd}
\end{equation*}
this is, $\otimes_{*}(\alpha+x,\beta+y)=\psi(\alpha,\beta)+\phi_{*}(x,y)$ for $\alpha+x \in \Z/m\oplus \Z$ and $\beta+y \in \Z/n\oplus \Z$. By the Eckmann-Hilton argument and Corollary \ref{cTP}, $\otimes_{*}(\alpha+x,\beta+y)=\psi(\alpha,\beta)+(bnx+amy)$.
\end{proof}

\begin{claim}\label{chc}
Diagram \eqref{hc} commutes up to a pointed homotopy,
\begin{equation}\label{hc}
\begin{tikzcd}[column sep=huge]
S^{1}\times S^{1} \arrow[r,"t\times t"] \arrow[d,"\phi"] & \U_{am}/\mu_{m}\times\U_{bn}/\mu_{n} \arrow[d,"\otimes"] \\
S^{1} \arrow[r,"t"] & \U_{abmn}/\mu_{mn}.
\end{tikzcd}
\end{equation}
\end{claim}
\begin{proof}
Consider the stabilization maps $\s_{j}:\U_{1} \to \U_{abmn}$ for $j=1,\dots,abmn$. Let $\upsilon, \omega \in S^{1}$, then
\[
t(\phi(\upsilon,\omega))=%
\begin{pmatrix}
\upsilon^{bn}\omega^{am} & 0 \\
0 & I_{abmn-1}
\end{pmatrix}%
\;\; \text{and} \;\;%
t(\upsilon)\tensor t(\omega)=%
\begin{pmatrix}
\upsilon t(\omega) & 0 & \cdots & 0\\
0 & t(\omega) & \cdots & 0 \\
\vdots & \vdots & \ddots & \vdots\\
0 & 0 & \cdots & t(\omega)
\end{pmatrix}.
\]

Observe that $t\circ \phi$ is equal to the composite
\begin{equation*}
\begin{tikzcd}[row sep=large]
S^{1}\times S^{1} \arrow[r,"\Delta \times \Delta"] & (S^{1})^{\times bn}\times (S^{1})^{\times am} \arrow[r,"g"] & \U_{abmn}^{\times bn}\times\U_{abmn}^{\times am} \arrow[r,"\m"]   & \U_{abmn}
\end{tikzcd}
\end{equation*}
where $g=(\s_{1}\times \cdots \times \s_{1},\s_{1}\times \cdots \times \s_{1})$, and $t\tensor t$ is equal to
\begin{equation*}
\begin{tikzcd}
S^{1}\times S^{1} \arrow[r,"\Delta \times \Delta"] & (S^{1})^{\times bn}\times (S^{1})^{\times am} \arrow[r,"h"] & \U_{abmn}^{\times bn}\times\U_{abmn}^{\times am} \arrow[r,"\m"]   & \U_{abmn}.
\end{tikzcd}
\end{equation*}
where $h=(\s_{1}\times \s_{2}\times \cdots \times \s_{bn},\s_{1}\times \s_{bn+1}\times \cdots \times \s_{(am-1)bn+1})$. By Lemma \ref{sj}, $t\circ \phi$ and $t\tensor t$ are pointed homotopic.
\end{proof}

\section{Proof of Theorem \ref{main}}

\begin{proposition}\label{Ttildehg}
Let $a$, $b$, $m$ and $n$ be positive integers such that $am$ and $bn$ are relatively prime and $am<bn$. Then there exist positive integers $u$ and $v$ satisfying $|vn(bn)^{n}-um(am)^{m}|=1$, so that there exist a positive integer $N$ and a homomorphism $\mif{\Tr}{\U_{am}\times \U_{bn}}{\U_{N}}$ such that 
\begin{enumerate}
\item the homomorphism $\Tr$ factors through $\mif{\widetilde{\Tr}}{\U_{am}/\mu_{m}\times \U_{bn}/\mu_{n}}{\U_{N}}$, and

\item the homomorphisms induced on homotopy groups 
\[
\mif{\widetilde{\Tr}_{i}}{\pi_{i}(\U_{am}/\mu_{m})\times \pi_{i}(\U_{bn}/\mu_{n})}{\pi_{i}(\U_{N})}
\]
are given by
\begin{equation*}
\begin{cases}
\widetilde{\Tr}_{i}(x,y)=um(am)^{m-1}x+vn(bn)^{n-1}y & \text{if $1<i<2am$,}\\
\widetilde{\Tr}_{i}(\alpha+x,\beta+y)=um(am)^{m-1}x+vn(bn)^{n-1}y & \text{if $i=1$,}
\end{cases}
\end{equation*}
where $\alpha \in \Z/m$, $\beta \in \Z/n$ and $x,y \in \Z$.
\end{enumerate}
\end{proposition}
\begin{proof}
We first construct $\Tr$.

Since $am$ and $bn$ are relatively prime, so are $m(am)^{m}$ and $n(bn)^{n}$. Hence there exist positive integers $u$ and $v$ such that $vn(bn)^{n}-um(am)^{m}=\pm1$. Let $N$ denote $u(am)^{m}+v(bn)^{n}$. We define $\Tr$ using the operations described in Section 2, as the composite
\begin{equation*}
\begin{tikzcd}[column sep=large]
\U_{am}\times\U_{bn} \arrow[r,"(\tensor^{m}\text{,}\tensor^{n})"] & \U_{(am)^{m}}\times\U_{(bn)^{n}} \arrow[r,"(\oplus^{u}\text{,}\oplus^{v})"] & \U_{u(am)^{m}}\times\U_{v(bn)^{n}} \arrow[r,"\oplus"]   & \U_{N}.
\end{tikzcd}
\end{equation*}

\begin{enumerate}
\item We must show that $\mu_{m}\times \mu_{n}$ is contained in $\Ker(\Tr)$. Let $\alpha$ and $\beta$ be $m$-th and $n$-th roots of unity, respectively. Note that the element $\bigl(\alpha I_{am},\beta I_{bn}\bigr)$ is sent to $\bigl(I_{u(am)^{m}},I_{v(bn)^{n}}\bigr)$ by $(\tensor^{m},\tensor^{n})$, hence to the identity by the composite $\Tr$ defined above.

\item We first observe that Corollaries \ref{cDS}, \ref{cnBS} and \ref{cnBT} imply 
\begin{equation*}
\begin{tikzcd}[row sep=tiny]
\Tr_{i}: \pi_{i}(\U_{m})\times\pi_{i}(\U_{n}) \arrow[r] & \pi_{i}(\U_{N})\\
\qquad(x,y) \arrow[r,mapsto] & um(am)^{m-1}x+vn(bn)^{n-1}y
\end{tikzcd}
\end{equation*}
for all $i<2am$.

From part (1) there is a map of fibrations
\begin{equation*}
\begin{tikzcd}
\mu_{m}\times\mu_{n} \arrow[r,hookrightarrow] \arrow[d] & \U_{am}\times \U_{bn} \arrow[r] \arrow[d,"\Tr"] & \U_{am}/\mu_{m}\times\U_{bn}/\mu_{n} \arrow[d,"\widetilde{\Tr}"]\\
\{I_{N}\} \arrow[r] & \U_{N} \arrow[r,equal] & \U_{N}.
\end{tikzcd}
\end{equation*}

\textbf{Case 1.} Let $i>1$. From the long exact sequence, diagram \eqref{iabove1} commutes.
\begin{equation}\label{iabove1}
\begin{tikzcd}
\pi_{i}(\U_{am})\times\pi_{i}(\U_{bn}) \arrow[r,"\iso"] \arrow[d,"\Tr_{i}"] & \pi_{i}( \U_{am}/\mu_{m})\times\pi_{i}(\U_{bn}/\mu_{n}) \arrow[d,"\widetilde{\Tr}_{i}"]\\
\pi_{i}(\U_{N}) \arrow[r,equal] & \pi_{i}(\U_{N}).
\end{tikzcd}
\end{equation}
Thus, $\widetilde{\Tr}_{i}(x,y)=\Tr_{i}(x,y)=um(am)^{m-1}x+vn(bn)^{n-1}y$ for $1<i<2m$.

\textbf{Case 2.} Let $i=1$. From the long exact sequence there is a homomorphism of short exact sequences
\begin{equation*}
\begin{tikzcd}[column sep=scriptsize]
\pi_{1}(\U_{am})\times \pi_{1}(\U_{bn}) \arrow[r,hookrightarrow] \arrow[d,"\Tr_{1}"] & \pi_{1}(\U_{am}/\mu_{m})\times\pi_{1}(\U_{bn}/\mu_{n}) \arrow[r,twoheadrightarrow] \arrow[d,"\widetilde{\Tr}_{1}"] & \Z/m\times \Z/n \arrow[d]\\
\pi_{1}(\U_{N}) \arrow[r,equal] & \pi_{1}(\U_{N}) \arrow[r] & 0.
\end{tikzcd}
\end{equation*}

The top short exact sequence splits. By direct inspection we obtain $\widetilde{\Tr}_{1}(\alpha+x,\beta+y)=\Tr_{1}(x,y)=um(am)^{m-1}x+vn(bn)^{n-1}y$.
\end{enumerate}
\end{proof}

\subsection{A left homotopy inverse}

Let $a$, $b$, $m$ and $n$ be positive integers. By applying the classifying-space functor to the homomorphism \eqref{q}, we obtain a map 
\begin{align}\label{SU}
F_{\tensor}: \B\U_{am}/\mu_{m}\times\B\U_{bn}/\mu_{n} \rightarrow \B\U_{abmn}/\mu_{mn}.
\end{align}
If we take the quotient by $\mu_{am}$ and $\mu_{bn}$ in \eqref{SU}, we write $f_{\tensor}$ instead of $F_{\tensor}$.

Let $J$ be the map
\begin{equation*}
\begin{tikzcd}[row sep=tiny]
J:\B\U_{am}/\mu_{m}\times\B\U_{bn}/\mu_{n} \arrow[r] & \B\U_{abmn}/\mu_{mn}\times\B\U_{N}\\
\qquad(x,y) \arrow[r,mapsto] & \Bigl(F_{\tensor}(x,y),\B\widetilde{\Tr}(x,y)\Bigr),
\end{tikzcd}
\end{equation*}
where the integer $N$ is the one provided by Proposition \ref{Ttildehg}.

\begin{proposition}\label{UN}
Let $a$, $b$, $m$ and $n$ be positive integers such that $am$ and $bn$ are relatively prime and $am<bn$. The map $J$ is $(2am+1)$-connected.
\end{proposition}
\begin{proof}
We want to prove that the induced homomorphism on homotopy groups
\begin{equation}\label{FB}
\begin{tikzcd}[row sep=large]
\pi_{i}(\B\U_{am}/\mu_{m})\times \pi_{i}(\B\U_{bn}/\mu_{n}) \arrow[d,"J_{i}"]\\
\pi_{i}(\B\U_{abmn}/\mu_{mn})\times\pi_{i}(\B\U_{N})
\end{tikzcd}
\end{equation}
is an isomorphism for all $i<2am+1$ and an epimorphism for $i=2am+1$.

Observe that the homotopy groups of the spaces involved are trivial in odd degrees below $2am+2$, hence it suffices to prove that $J_{i}$ is an isomorphism for all $i$ even and $i<2am+1$. 

We divide the proof into two cases.

\textbf{Case 1.} Let $i<2am+1$ and $i\neq 2$. For this case computations can be done at the level of the universal covers of the groups $\U_{am}/\mu_{m}, \U_{bn}/\mu_{n}$ and $\U_{abmn}/\mu_{mn}$. 

The homomorphism \eqref{FB} takes the form 
\[
\mif{J_{i}}{\Z\times\Z}{\Z\times\Z}.
\]
Propositions \ref{TP1} and \ref{Ttildehg} yield 
\[
J_{i}(x,y) =\Bigl(bnx+amy,um(am)^{m-1}x+vn(bn)^{n-1}y\Bigr).
\]
Thereby, the homomorphism \eqref{FB} is represented by the matrix
\[
\begin{pmatrix}
bn & am\\
m(am)^{m-1}u & n(bn)^{n-1}v
\end{pmatrix},
\]
which is invertible. This proves $J_{i}$ is an isomorphism.

\textbf{Case 2.} Let $i=2$. The homomorphism \eqref{FB} takes the form 
\[
\mif{J_{2}}{(\Z/m\oplus\Z)\times(\Z/n\oplus\Z)}{(\Z/mn\oplus\Z)\times\Z}.
\]

Propositions \ref{TP1} and \ref{Ttildehg} yield 
\begin{align*}
J_{2}(x+\alpha,y+\beta) &= \Bigl(\psi(\alpha,\beta)+(bnx+amy),um(am)^{m-1}x+vn(bn)^{n-1}y\Bigr).
\end{align*}

Recall that $\psi:\Z/m \times \Z/n \to \Z/mn$ is addition where $\Z/m$ and $\Z/n$ are considered as subgroups of $\Z/mn$, see proof of Proposition \ref{TP1}. The homomorphism $\psi$ is an isomorphism. From this and the invertibility of the matrix above, $J_{2}$ is an isomorphism.
\end{proof}
\subsection{Factorization through \texorpdfstring{$F_{\tensor}: \B\U_{am}/\mu_{m}\times\B\U_{bn}/\mu_{b} \rightarrow \B\U_{abmn}/\mu_{mn}$}{F:BUam/m x BUbn/n-->BUabmn/mn}}

\begin{theorem}\label{abpq/pq}
Let $a$, $b$, $m$ and $n$ be positive integers such that $am$ and $bn$ are relatively prime and $am<bn$. Let $X$ be a topological space with the homotopy type of a finite dimensional CW complex such that $\dim(X)\leq 2am+1$.

Every map $\Az{A}:X \rightarrow \B\U_{abmn}/\mu_{mn}$ can be lifted to $\B\U_{am}/\mu_{m}\times\B\U_{bn}/\mu_{n}$ along the map $F_{\tensor}$  up to a pointed homotopy.
\end{theorem}
\begin{proof}
Diagramatically speaking, we want to find a map 
\[
\Az{A}_{m}\times\Az{A}_{n}:X \rightarrow \B\U_{am}/\mu_{m}\times\B\U_{bn}/\mu_{n}
\]
 such that diagram \eqref{lpF} commutes up to homotopy
\begin{equation}\label{lpF}
\begin{tikzcd}[execute at begin picture={\useasboundingbox (-4.5,-1) rectangle (4.5,1);},row sep=large,column sep=huge]
& \B\U_{am}/\mu_{m}\times\B\U_{bn}/\mu_{n} \arrow[d,"F_{\tensor}"] \\
 X \arrow[r,"\Az{A}"] \arrow[ur,dotted,bend left=20,"\Az{A}_{m}\times\Az{A}_{n}"]  & \B\U_{abmn}/\mu_{mn}.
\end{tikzcd}
\end{equation}

Proposition \ref{UN} yields a map $\mif{J}{\B\U_{am}/\mu_{m}\times\B\U_{bn}/\mu_{n}}{\B\U_{abmn}/\mu_{mn}\times\B\U_{N}}$ where $N$ is some positive integer. Observe that $F_{\tensor}$ factors through $\B\U_{abmn}/\mu_{mn}\times\B\U_{N}$, so we can write $F_{\tensor}$ as the composite of $J$ and the projection $\proj_{1}$ shown in diagram \eqref{FJ}.

\begin{equation}\label{FJ}
\begin{tikzcd}[row sep=large,column sep=huge]
\B\U_{am}/\mu_{m}\times\B\U_{bn}/\mu_{n} \arrow[r,"J"] \arrow[d,"F_{\tensor}"] & \B\U_{abmn}/\mu_{mn}\times\B\U_{N} \\
\B\U_{abmn}/\mu_{mn} \arrow[ur,leftarrow,"\proj_{1}"]& 
\end{tikzcd}
\end{equation}

Since $J$ is $(2am+1)$-connected and $\dim(X)\leq2am+1$, then by Whitehead's theorem
\[
J_{\#}:[X, \B\U_{am}/\mu_{m}\times\B\U_{bn}/\mu_{n}] \rightarrow [X,\B\U_{abmn}/\mu_{mn}\times\B\U_{N}]
\]
is a surjection, \cite[Corollary 7.6.23]{SpaAT2012}. 

Let $s$ denote a section of $\proj_{1}$. The surjectivity of $J_{\#}$ implies $s\circ \Az{A}$ has a preimage $\Az{A}_{m}\times\Az{A}_{n}:X \rightarrow \B\U_{am}/\mu_{m}\times\B\U_{bn}/\mu_{n}$ such that $J\circ (\Az{A}_{m}\times\Az{A}_{n})\simeq s\circ \Az{A}$.

The commutativity of diagram \eqref{lpF} follows from commutativity of diagram \eqref{FJ}. Thus, the result follows.
\end{proof}

\subsection{Factorization through \texorpdfstring{$f_{\tensor}: \B\PU_{am}\times\B\PU_{bn} \rightarrow \B\PU_{abmn}$}{f:BPUam x BPUbn-->BPUabmn}}

\begin{proposition}\label{Rel}
Let $X$ be a finite CW complex. Let $\alpha \in \Br(X)$ be a class of period $m$. There exists a lifting $\Az{A}'$ of $\alpha$ if and only if $\alpha$ is represented by a topological Azumaya algebra $\Az{A}$ of degree $am$.
\end{proposition}
\begin{proof}
Let $\alpha \in \Br(X)$ be a Brauer class of period $m$. There exists a lifting $\xi \in \Hy^{2}(X;\Z/m)$ such that $\RBn_{m}(\xi)=\alpha$. Diagrammatically,
\begin{equation*}
\begin{tikzcd}[row sep=large]
  & X \arrow[dl, rightarrow,"\xi" above,dotted,bend right] \arrow{d}{\alpha} & \\
 \K(\Z/m,2)\arrow[r,"\RBn_{m}"] & \K(\Z,3) \arrow{r}{\times m} & \K(\Z,3).
\end{tikzcd}
\end{equation*}
The map of fibrations below 
\begin{equation*}
\begin{tikzcd}[column sep=large]
\mu_{m} \arrow[r,hookrightarrow] \arrow[d,hookrightarrow]& \U_{am} \arrow[r,hookrightarrow] \arrow[d,equal] & \U_{am}/\mu_{m}\arrow[d,"q"]\\
S^{1} \arrow[r,hookrightarrow]  & \U_{am} \arrow[r] & \PU_{am}
\end{tikzcd}
\end{equation*}
induces a commutative diagram
\begin{equation*}\label{mapofP}
\begin{tikzcd}[column sep=large]
\B\U_{am}/\mu_{m} \arrow[r,"\B q"] \arrow{d} & \B\PU_{am} \arrow{d}\\
\K(\Z/m,2) \arrow[r,"\RBn_{m}"] & \K(\Z,3).
\end{tikzcd}
\end{equation*}

In order to prove the proposition, we show that there exists a lifting $\Az{A}'$ of $\xi$ if and only if there exists a lifting $\Az{A}$ of $\xi$, see diagram \eqref{hpPt} below. 
\begin{equation}\label{hpPt}
\begin{tikzcd}[column sep=large, row sep=large]
& \B\U_{am}/\mu_{m} \arrow[r,"\B q"] \arrow{d} & \B\PU_{am} \arrow[lldd,"\Az{A}",leftarrow,bend left=18,dotted] \arrow{d}\\
& \K(\Z/m,2) \arrow[r,equal] \arrow[d,equal] & \K(\Z/m,2) \arrow{d}{\RBn_{m}}\\
X \arrow[ruu,"\Az{A}'", bend left,dotted] & \K(\Z/m,2) \arrow[l,leftarrow,"\xi"] & \K(\Z,3) \arrow[l,"\RBn_{m}",leftarrow]
\end{tikzcd}
\end{equation}

If there exists a lifting $\mif{\Az{A}'}{X}{\B\U_{am}/\mu_{m}}$, then the composite $\B q\circ \Az{A}'$ is a topological Azumaya algebra of degree $am$ that represents the Brauer class $\alpha$.

Conversely, suppose there exists an Azumaya algebra $\Az{A}$ of degree $am$ making the outer square in the diagram below commute up to homotopy.
\begin{equation*}
\begin{tikzcd}
X \arrow[rrd, bend left,  "\Az{A}" ] \arrow[ddr, bend right, "\xi'" left] \arrow[rd, rightarrow, "\Az{A}'",dotted] & & \\
 & \B\U_{am}/\mu_{m} \arrow[r,"\B q"] \arrow{d} & \B\PU_{am} \arrow{d}\\
 & \K(\Z/m,2) \arrow[r,"\RBn_{m}"] & \K(\Z,3) \\
\end{tikzcd}
\end{equation*}

In the inner square, the induced map on the homotopy fibers of $\B q$ and $\RBn_{m}$ is a homotopy equivalence. An application of the 5-lemma implies that the inner square is a homotopy pullback square. Therefore, there exists a lifting $\Az{A}'$ representing $\alpha$.
\end{proof}

\begin{theorem}\label{abpq}
Let $a$, $b$, $m$ and $n$ be positive integers such that $am$ and $bn$ are relatively prime and $am<bn$. Let $X$ be a CW complex such that $\dim(X)\leq 2am+1$. 

If $\Az{A}$ is a topological Azumaya algebra of degree $abmn$ such that $\cl(\Az{A})$ has period $mn$, then there exist topological Azumaya algebras $\Az{A}_{m}$ and $\Az{A}_{n}$ of degrees $am$ and $bn$, respectively, such that $\per(\cl(\Az{A}_{m}))=m$, $\per(\cl(\Az{A}_{n}))=n$ and $\Az{A}\iso \Az{A}_{m}\tensor\Az{A}_{n}$.
\end{theorem}
\begin{proof}
In this case we want to solve the lifting problem shown in diagram \eqref{MLP} up to homotopy, with $\per(\cl(\Az{A}_{m}))=m$, $\per(\cl(\Az{A}_{n}))=n$. 
\begin{equation}\label{MLP}
\begin{tikzcd}[row sep=large, column sep=huge]
 & \B\PU_{am}\times\B\PU_{bn} \arrow{d} \arrow[d,"f_{\tensor}"] \\
X \arrow[r,"\Az{A}"]  \arrow[ur,dotted,"\Az{A}_{m}\times\Az{A}_{n}"] & \B\PU_{abmn}.
\end{tikzcd}
\end{equation}

By Proposition \ref{Rel} there exists a map $\Az{A}':X \rightarrow \B\U_{abmn}/\mu_{mn}$ such that $\per(\cl(\Az{A}'))=\per(\cl(\Az{A}))=mn$. Then, by Theorem \ref{abpq/pq} there exists a map $\mif{\Az{A}'_{m}\times\Az{A}'_{n}}{X}{\B\U_{am}/\mu_{m}\times\B\U_{bn}/\mu_{n}}$ such that $F_{\tensor}\circ(\Az{A}'_{m}\times\Az{A}'_{n})\simeq\Az{A}'$.

Since $\cl(\Az{A}'_{m})\cl(\Az{A}'_{n})=\cl(\Az{A}'_{mn})$, $m$ and $n$ are relatively prime, and $\per(\cl(\Az{A}'))=mn$ then $\per(\cl(\Az{A}'_{m}))=m$ and $\per(\cl(\Az{A}'_{n}))=n$.

By Proposition \ref{Rel} there exists a map $\mif{\Az{A}_{m}\times\Az{A}_{n}}{X}{\B\PU_{am}\times\B\PU_{bn}}$ such that $\per(\cl(\Az{A}_{m}))=\per(\cl(\Az{A}'_{m}))$ and $\per(\cl(\Az{A}_{n}))=\per(\cl(\Az{A}'_{n}))$.

It remains to show that diagram (\ref{MLP}) commutes. Consider the diagram below
\begin{equation}\label{ss}
\begin{tikzcd}[row sep=normal,column sep=huge]
\B\U_{am}/\mu_{m}\times\B\U_{bn}/\mu_{n} \arrow[rd,leftarrow,"\Az{A}'_{m}\times\Az{A}'_{n}"] \arrow[rr,"\B q\times\B q"] & &\B\PU_{am}\times\B\PU_{bn}  \arrow[dd,"f_{\tensor}"]\\
& X \arrow[ru,"\Az{A}_{m}\times\Az{A}_{n}"] \arrow[rd,"\Az{A}"]& \\
\B\U_{abmn}/\mu_{mn} \arrow[uu,leftarrow,"F_{\tensor}"] \arrow[ru,leftarrow,"\Az{A}'"] \arrow[rr] & & \B\PU_{abmn}
\end{tikzcd}
\end{equation}
Observe that the square, as well as top, bottom and left triangles, of diagram \eqref{ss} commute. Hence, the right triangle commutes.
\end{proof}

Theorem \ref{main} is a corollary of Theorem \ref{abpq}.

\begin{theorem}\label{boun}
Let $a$, $b$, $m$ and $n$ be positive integers such that $am$ and $bn$ are relatively prime and $am<bn$. The map $F_{\tensor}: \B\U_{am}/\mu_{m}\times\B\U_{bn}/\mu_{n} \rightarrow \B\U_{abmn}/\mu_{mn}$ does not have any section.
\end{theorem}
\begin{proof}
Suppose there exists a section $\sigma$ of $F_{\tensor}$. 

By Proposition \ref{TP1} the map $F_{\tensor}$ induces a homomorphism on homotopy groups which is given by $(x,y)\mapsto bn\esta(x)+am\esta(y)$ for $i>2$. In degree $2am+2$ the homomorphism $(F_{\tensor})_{*}$ takes the form $(F_{\tensor})_{*}:\pi_{2am+2}(\B\U_{am}/\mu_{m})\times \Z  \rightarrow \Z$, where $\pi_{2am+2}(\B\U_{am}/\mu_{m})\iso \pi_{2am+2}(\B\U_{am})$ and $\pi_{2am+2}(\B\U_{am})$ is trivial when $am$ is odd, and $\Z/2$ when $am$ is even, see \cite[Page 971]{M1995}. Therefore, $(F_{\tensor})_{*}(x,y)=am\esta(y)$. Thus $\im(F_{\tensor})_{*}=am\Z$.

On the other side, since $\sigma$ is a section of $F_{\tensor}$, the composite $(F_{\tensor})_{*}\circ \sigma_{*}$
is the identity. This contradicts the fact that $\im( (F_{\tensor})_{*}\circ \sigma_{*}) \subset am\Z$.
\end{proof}

In Theorem \ref{main} it is proven that there exists a tensor product decomposition for topological Azumaya algebras over low dimensional CW complexes, and that such decomposition does not exist for topological Azumaya algebras over an arbitrary CW complex. The proof of Theorem \ref{boun} implies that for positive integers $m$ and $n$ where  $m<n$, if $\Az{A}$ is a topological Azumaya algebra of degree $mn$ over a finite CW complex of dimension higher than $2m+1$, then $\Az{A}$ may not be decomposable as $\Az{A}_{m}\tensor \Az{A}_{n}$. In fact, consider the unit $(2m+2)$-sphere, and let $\Az{S}:S^{2m+2}\to \B\PU_{mn}$ be a degree-$mn$ topological Azumaya algebra on $S^{2m+2}$ such that $\Az{S}$ generates $\pi_{2m+2}(\B\PU_{mn})$, then $\Az{S}$ cannot be decomposed as the tensor product of topological Azumaya algebras of degrees $m$ and $n$.

\begin{remark}\label{notuniq}
Under the hypotheses of Theorem \ref{main}, the topological Azumaya algebras $\Az{A}_{m}$ and $\Az{A}_{n}$ are not neccesarily unique up to isomorphism. In order to see this, we consider the Moore-Postnikov tower for $f_{\tensor}$:
\begin{equation*}
\begin{tikzcd}[row sep=small]
& \B\PU_{m}\times\B\PU_{n} \arrow{d} &\\
& \vdots \arrow{d} &\\
\K\bigl(\pi_{2}F,2\bigr) \arrow[r] & Y[2] \arrow{d} \\
\K\bigl(\pi_{1}F,1\bigr) \arrow[r] & Y[1]  \arrow[d]  \arrow[r,"k_{1}"] & \K(\pi_{2}F,3)\\
& \B\PU_{mn}  \arrow[r,"k_{0}"]& \K(\pi_{1}F,2),
\end{tikzcd}
\end{equation*}
where $F$ is the homotopy fiber of $f_{\tensor}$, and $k_{i-1}:Y[i-1]\rightarrow \K\bigl(\pi_{i}F,i+1\bigr)$ is the $k$-invariant that classifies the fiber sequence $Y[i] \rightarrow Y[i-1]$, for $i>0$, \cite[Theorem 4.71]{Hatch2002}.

Since the map $f_{\tensor}$ induces an isomorphism on $\pi_{2}$, and $\pi_{2i+1}$ for $0<i<m$, it follows that $\B\PU_{mn} \simeq Y[i]$ for $i=1,2,3$, and $Y[2i]\simeq Y[2i+1]$ for $1<i<m$.

The long exact sequence of $F\to \B\PU_{m}\times\B\PU_{n} \to \B\PU_{mn}$ yields 
\begin{equation*}
\pi_{i}F\iso%
\begin{cases}
0 & \text{if $i=2$ or $i$ is odd and $i<2m+1$},\\
\Z & \text{if $i\neq2$, $i$ is even and $i<2m+1$}.
\end{cases}
\end{equation*}

Hence the Moore-Postnikov tower of $f_{\tensor}$ takes the form
\begin{equation*}\label{Pt}
\begin{tikzcd}[row sep=small]
& \B\PU_{m}\times\B\PU_{n} \arrow{d} &\\
& \vdots \arrow{d} &\\
\K(\Z,6) \arrow[r] & Y[6] \arrow{d} \\
\K(\Z,4) \arrow[r] & Y[4]  \arrow[d]  \arrow[r,"k_{5}"] & \K(\Z,7)\\
& \B\PU_{mn}  \arrow[r,"k_{3}"]& \K(\Z,5)
\end{tikzcd}
\end{equation*}

Let $X$ be a CW complex of $\dim(X)\leq 6$. Let $m$ and $n$ be as in the hypothesis of Theorem \ref{main}, and $m>3$. Let $\Az{A}$ be a topological Azumaya algebra of degree $mn$.

Observe that the $k$-invariant $k_{3}$ is trivial because $\Hy^{5}(\B\PU_{mn};\Z)$ is trivial, \cite[Proposition 4.1]{AW6com2013}. Hence there is no obstruction to lift $\Az{A}$ to $Y[4]$. Similarly, we can lift the identity map $\id_{\B\PU_{mn}}$ to $Y[4]$, in this case we obtain the splitting $Y[4]\simeq \B\PU_{mn}\times\K(\Z,4)$. Then the lifting of $\Az{A}$ takes the form $(\Az{A},\xi):X\rightarrow \B\PU_{mn}\times\K(\Z,4)$.

The cohomology groups of $X$ vanish for all degrees greater than 6, given that $X$ is 6-dimensional. Thus $(\Az{A},\xi)$ can be lifted up the Moore-Postnikov tower to $\B\PU_{m}\times\B\PU_{n}$. See diagram \eqref{Pt2}.

\begin{equation}\label{Pt2}
\begin{tikzcd}[row sep=small]
&&\B\PU_{m}\times\B\PU_{n} \arrow{d} & \\
&& \vdots \arrow{d} & \\
&& Y[4]\simeq\B\PU_{mn}\times \K(\Z,4) \arrow{d} \arrow[r,"k_{5}"] & \K(\Z,7)\\
X \arrow[rr,"\Az{A}"] \arrow[rru,bend left=10,"(\Az{A}\text{,}\xi)" near end] \arrow[rruuu,bend left,"\Az{A}_{m}\times\Az{A}_{n}"] & &\B\PU_{mn} \arrow[r,"k_{3}"] & \K(\Z,5)\\
\end{tikzcd}
\end{equation}

This proves that $\Az{A}$ can be decomposed as $\Az{A}_{m}\tensor\Az{A}_{n}$. The lifting $(\Az{A},\xi)$ is not necessarily unique. In fact, every cohomology class $\xi \in \Hy^{4}(X;\Z)$ gives rise to a lifting $(\Az{A},\xi)$. 
\end{remark}

\bibliographystyle{alpha}
\bibliography{NinyDecompAzu}
\end{document}